\documentclass[a4paper]{amsart}
  \usepackage{amssymb,exscale,bbm,mathrsfs}
  \usepackage[utf8]{inputenc}

 \newtheorem{theorem}{Theorem}[section]

 \theoremstyle{definition}
 \newtheorem{definition}[theorem]{Definition}
 \theoremstyle{remark}

 \numberwithin{equation}{section}

\begin{document}

%

%
  \renewcommand {\Re}{\text{Re}}
  \renewcommand {\Im}{\text{Im}}
  \newcommand {\upi}{{\mathrm \pi}}
  \newcommand {\sfrac}[2] { {{}^{#1}\!\!/\!{}_{#2}}} 
  \newcommand {\pihalbe}{\sfrac{\upi}2}
  \newcommand {\einhalb}{\sfrac{1}2}
  \newcommand {\Borel}{{\mathcal Bo}}
  \newcommand {\Poisson}{{\mathscr P}}
  \newcommand {\cH}{\mathscr H}  
  \newcommand {\cY}{\mathscr Y}  
  \newcommand {\CC}{\mathbb C}  
  \newcommand {\EE}{\mathbb E}
  \newcommand {\NN}{\mathbb N}
  \newcommand {\PP}{\mathbb P}
  \newcommand {\ZZ}{\mathbb Z}
  \newcommand {\RR}{\mathbb R}
  \newcommand {\cT}{\mathcal T}
  \newcommand {\BOUNDED}{\mathcal B}
  \newcommand {\DOMAIN}{\mathcal D}
  \newcommand {\FOURIER}{\mathcal F}
  \newcommand {\LAPLACE}{\mathcal L}
  \newcommand {\HOERMANDER}[1]{{\mathcal H}^{#1}}
  \newcommand {\WEIGHT}[1]{ {\langle{#1}\rangle} }
  \newcommand {\eins} {\mathbbm 1}
  \newcommand {\al}{\alpha}
  \newcommand {\la}{\lambda}
  \newcommand {\eps}{\varepsilon}
  \newcommand {\Ga}{\Gamma}
  \newcommand {\ga}{\gamma}
  \newcommand {\om}{\omega}
  \newcommand {\Om}{\Omega}
  \newcommand {\Xm}{X_{-1}}
  \newcommand {\SECT}[1] {S_{{#1}}}
  \newcommand {\STRIP}[1] {{\rm St}_{{#1}}}
  \newcommand {\norm}[1] {\| #1 \|}  
  \newcommand {\abs}[1] {|#1|}
  \newcommand {\biggabs}[1] {\bigg|#1\bigg|}
  \newcommand {\lrnorm}[1]{\left\| #1 \right\|}
  \newcommand {\bignorm}[1]{\bigl\| #1 \bigr\|}
  \newcommand {\Bignorm}[1]{\Bigl\| #1 \Bigr\|}
  \newcommand {\Biggnorm}[1]{\Biggl\| #1 \Biggr\|}
  \newcommand {\biggnorm}[1]{\biggl\| #1 \biggr\|}
  \newcommand {\Bigidual}[3] {\Bigl\langle #1, #2 \Bigr\rangle_{#3}}
  \newcommand {\bigidual}[3] {\bigl\langle #1, #2 \bigr\rangle_{#3}}
  \newcommand {\idual}[3] {\langle #1, #2 \rangle_{#3}}
  \newcommand {\Bigdual}[2] {\Bigidual{#1}{#2}{}}
  \newcommand {\bigdual}[2] {\bigidual{#1}{#2}{}}
  \newcommand {\dual}[2] {\idual{#1}{#2}{} }
  \newcommand {\SP}[2]{ [#1 | #2] }
  \newcommand {\weg} {\backslash}
  \newcommand {\SUCHTHAT}{:\;}
  \newcommand {\emb} {\hookrightarrow}
  \newcommand {\phitilde}{ {\widetilde \phi}}
\renewcommand{\labelenumi} {(\alph{enumi})}    
\renewcommand{\labelenumii}{(\roman{enumii})}
\renewcommand{\theenumi} {(\alph{enumi})}      
\renewcommand{\theenumii}{(\roman{enumii})}    
  \newcommand{\ud}{{\mathrm d}}

\allowdisplaybreaks

\title[The Weiss conjecture and weak norms]{The Weiss conjecture and weak norms}
\author[Bernhard H. Haak]{Bernhard H. Haak}
\address{%
Institut de Math\'ematiques de Bordeaux\\Universit\'e Bordeaux 1\\351, cours de
la Lib\'eration\\33405 Talence CEDEX\\FRANCE}
\email{bernhard.haak@math.u-bordeaux1.fr}
\thanks{This work was partially supported by the ANR project ANR-09-BLAN-0058-01}
\date{\today}

\begin{abstract}
  In this note we show that for analytic semigroups the so-called Weiss
  condition of uniform boundedness of the operators  
\[
  Re(\la)^\einhalb C(\la+A)^{-1}, \qquad Re(\la)>0
\]
on the complex right half plane and weak Lebesgue
$L^{2,\infty}$--admissibility are equivalent. Moreover, we show that
the weak Lebesgue norm is best possible in the sense that it is the
endpoint for the 'Weiss conjecture' within the scale of Lorentz spaces
$L^{p,q}$.
\end{abstract}

\maketitle

\section{Introduction}
In this note we are concerned with linear control systems of the form
\begin{equation}\label{eq:control-system}
\left\{
\begin{array}{lcl}
 x'(t)+A x(t) & = & 0\quad (t>0),\\
 x(0) & = & x_0,\\
 y(t) & = & C x(t)\quad (t>0),
\end{array}
\right.
\end{equation}
where $-A$ is the generator of a strongly continuous semigroup
$(T(\cdot))$ on a Banach space $X$. The function $x(\cdot)$ takes
values in $X$, and the function $y(\cdot)$ takes values in a 
Banach space $Y$. The observation operator
$C$ may be an unbounded operator from $X$ to $Y$. 
A commonly used minimal assumption on $C$ is
that $C$ is  bounded $X_1\to Y$ where $X_1$ denotes the domain $D(A)$
of $A$ equipped with the graph norm. We refer, e.g., to
\cite{Salamon, Staffans,Weiss:admiss-observation,
  Weiss:Admissibility-of-unbounded}. 

\begin{definition}
Let $\cY_\infty$ be a space of functions $\RR_+\to Y$ and let, for each
$\tau>0$, denote $\cY_\tau$ the restricted space of functions on
$[0,\tau]$. Then the system \eqref{eq:control-system} is called
$\cY_\tau$-\emph{admissible} for $\tau\in (0, \infty]$  if 
the output of the system \eqref{eq:control-system}
depends continuously on initial state, i.e. if the mapping
\[
  X \to  \cY_\tau, \qquad
  x_0  \mapsto y(\cdot)
\]
is continuous. 
\end{definition}

If $X$ and $Y$ are Hilbert spaces, a natural choice for $\cY_\tau$ is
$\cY_\tau = L^2([0, \tau], Y)$.  We mainly focus on infinite-time
admissibility, that is the case $\tau{=}\infty$. We will write
shorthand $L^2$--admissibility instead of $L^2(\RR_+;
Y)$--admissibility.  In other situations, other norms such as the
$L^p$--norm may be useful (see
e.g. \cite{HaakKunstmann:NavierStokes}).

In this note we discuss the failure of the Weiss conjecture. In order
to do so, we treat the case where $\cY$ is a Lorentz-Bochner space
$L^{p,q}(0,\tau; Y)$. We recall some basic definitions and properties
of these spaces. We refer to \cite{Grafakos,Hunt:Lorentz-spaces} for
references and further results if $Y=\CC$. The vector-valued spaces
are discussed e.g. in \cite{BlascoGregori:vector-valued-Lorentz}.
We recall the definitions.  Let $Y$ be a Banach space. For a
measurable, $Y$-valued function $f$ on a measure space $(\Omega,
\mu)$, define the distribution function
\[
d_f(\al) = \mu\left( \{ \om \in \Omega \SUCHTHAT 
          \norm{ f(\om) }>\al \} \right) 
\]
of $f$ and the non-decreasing rearrangement $f^*$ of $f$ as
\[
   f^*(t) = \inf\{ s>0 \SUCHTHAT  d_f(s) \le t \}.
\]
If $f$ is a continuous, real-valued, positive and decreasing function
on $[0, \infty)$, it is easy to see that $f^*= f$. For $1\le q\le
\infty$ and $p> 1$, the Lorentz-Bochner space $L^{p, q}(\Omega;
Y)$ is defined as the set of all measurable functions such that the
(quasi)-norm
\[
\norm{f}_{L^{p, q}(o, \tau; X)} 
= \left\{
  \begin{array}{ll} \displaystyle
    \left( \int_0^\infty \bigl( t^{\sfrac1p} f^*(t) \bigr)^q\, \tfrac{\ud t}t
    \right)^{\sfrac1q}  & \text{if } 1\le q< \infty \\
   \phantom{tatat} \\
    \sup\{\al^{\sfrac1p} d_f(\al) \SUCHTHAT \al>0\} & \text{if } q= \infty
  \end{array}
\right.
\]
is finite. If $Y=\CC$ we will simply write $L^{p,q}(\Omega)$.
Notice that by Fubinis theorem, $L^{p, p}(\Omega) = L^p(\Omega)$.  The
space $L^{p, \infty}$ is also called weak-Lebesgue space. A typical
weak-$L^p(\RR_+)$ function that is not in $L^p(\RR_+)$ is $f(x) =
|x|^{-\sfrac1p}$. Lorentz spaces are ``natural'' function spaces in
since they are real interpolation spaces between usual Lebesgue
spaces, see \cite{Triebel:interpolation}. As such, they appear
in the context of the Weiss conjecture, as will be explained
in the next section in detail.

\noindent Recall the definition of a sectorial operator: a densely
defined operator $A$ on a Banach space $X$ is called sectorial 
of angle $\om\in [0, \upi)$ if the spectrum of $A$ is contained in the
open sector $S_\om = \{ z\in \CC^*\SUCHTHAT |\arg(z)|<\om \}$ and if
for all larger angles $\theta \in (\om, \upi)$, the operators 
\[
  \{ \la (\la+A)^{-1} \SUCHTHAT   \la \not \in S_\theta \}
\]
are uniformly bounded. Negative generators of bounded
$C_0$--semigroups are sectorial of angle $\pihalbe$. Bounded analytic
$C_0$--semigroups are characterised by uniform boundedness of the operators
$tAT(t)$, $t>0$. They are precisely those semigroups whose (negative)
generator is sectorial of angle $<\pihalbe$ (see
e.g. \cite{EngelNagel} for details).

\section{The Weiss condition and decay rates of observed semigroup orbits}

\begin{definition}
Let $-A$ be the generator of a bounded strongly continuous semigroup
on a Banach space $X$ and $C\in \BOUNDED( \DOMAIN(A), Y)$ be an
observation operator. We say that $(A, C)$ satisfies the {\em Weiss
  condition} if the operators
\begin{equation}   \label{eq:weiss-condition}
  Re(\la)^\einhalb C(\la+A)^{-1}, \qquad Re(\la)>0
\end{equation}
are uniformly  bounded in $B(X, Y)$ when $\la$ runs through the open right half
plane.
\end{definition}

\noindent By the Laplace transform, 
\begin{equation}
  \label{eq:laplace-trafo}
    C (\la{+}A)^{-1}x = \int_0^\infty e^{-\la t}\; C T(t)x \, \ud t
\end{equation}
for $x\in \DOMAIN(A)$. If $C$ is $L^2$--admissible, i.e. if
$CT(\cdot)x \in L^2(\RR_+; Y)$, taking norms in
(\ref{eq:laplace-trafo}) and using Cauchy-Schwarz inequality
(\ref{eq:weiss-condition}) follows. The \emph{Weiss conjecture}
\cite{Weiss:conjectures} states that if $X$ and $Y$ are Hilbert
spaces, $L^2$--admissibility and the Weiss condition
(\ref{eq:weiss-condition}) are equivalent. This equivalence is known
to be true in a certain number of cases, for instance for normal
semigroups \cite{Weiss:conjectures} or semigroups of contractions with
scalar output \cite{JacobPartington:contraction}. However, the
conjecture is wrong, even in a restricted version were the output
space is $Y=\CC$ \cite{JacobZwart:counterexable-Weiss-conj}. For more
results and references concerning the Weiss conjecture we refer to the
survey \cite{JacobPartington:survey}.
Recall the following result of Le~Merdy.

 \begin{theorem}[ {\cite[Theorem 4.1]{LeMerdy:weiss-conj}}]
 Let $T(t)$ be a bounded analytic semigroup on a Banach space
 $X$. Assume that its generator $-A$ is injective and
 that $C$ satisfies the Weiss condition {\rm (\ref{eq:weiss-condition})}. 
 Then $C$ is (infinite time) $L^2$--admissible provided that
 $A$ admits upper square function estimates.
 \end{theorem}

 It is known \cite{LeMerdy:weiss-conj} that the extra assumption of
 upper square function estimates in Le~Merdy's theorem cannot dropped:
 it suffices to observe that $C=A^\einhalb$ satisfies
 {\rm (\ref{eq:weiss-condition})}, and that admissibility of $A^\einhalb$
 actually {\em is } an upper square function estimate.  What can be
 obtained instead of $L^2$--admissibility when dropping the assumption
 of upper square function estimates? We give an answer in the next two
 theorems.

\begin{theorem}\label{thm:schwache-L2-zulaessigkeit}
  Let $T(t)$ be an exponentially stable  analytic semigroup on a
  Banach space $X$. Then the following conditions are equivalent.
  \begin{enumerate}
  \item \label{item:weiss-cond} The Weiss condition  {\rm (\ref{eq:weiss-condition})}
  \item \label{item:decay-rate} There is a constant $K>0$ such that
\begin{equation}   \label{eq:norm-estimate}
    \norm{ CT(t) x}_Y   \le K t^{-\einhalb} \norm{ x}_X,\qquad t>0
\end{equation}
  \item \label{item:weak-admiss} $C$ is $L^{2,\infty}$ admissible. 
  \end{enumerate}
\end{theorem}

The next result tells us that the weak Lebesgue norm is optimal in the
sense that it is the endpoint for the 'Weiss conjecture' within the
scale of Lorentz spaces $L^{p,q}$.

\begin{theorem}\label{thm:optimality-counterexample}
  There exists an exponentially stable analytic semigroup $T(t)$ on a
  Hilbert space $H$ and a scalar valued observation operator $C$ such
  that  $CT(\cdot)x$ is not in any $L^{2,q}(0,\tau; Y)$ for whatsoever
  choice of $q<\infty$ or $\tau>0$.  
\end{theorem}

\begin{proof}[Proof of Theorem~\ref{thm:schwache-L2-zulaessigkeit}]
  \noindent We first show the equivalence of \ref{item:weiss-cond} and
  \ref{item:decay-rate}.  In \cite[Corollary 4.7]{HaakHaaseKunstmann}
  it is shown that the Weiss condition is equivalent to $C$ being
  bounded from $Z$ to $Y$, where $Z := (X, \dot X_1)_{\einhalb, 1}$ is
  the real interpolation space between $X$ and the homogeneous domain
  space $\dot X_1$ of $A$. Since $A$ is invertible, $\dot X_1 =
  X_1$. On the other hand, for analytic semigroups, the space $Z$ is
  characterised by the fact that
\[
 \norm{ C T(t) }_{Z \to Y} \lesssim t^{-\einhalb}. 
\]
This is essentially shown in \cite[Proposition
3.9]{HaakKunstmann:NavierStokes}. For the sake of completeness we
repeat the short argument.  Indeed, by analyticity of the semigroup,
\[
\norm{T(t)}_{X  \to \dot X_1} \lesssim t^{-1}
\quad\text{and}\quad
\norm{T(t)}_{\dot X_1 \to \dot X_1} \lesssim 1
\]
and from the boundedness of $C$ on $X_1$ we deduce
\[
\norm{CT(t)}_{X \to Y} \lesssim t^{-1}
\quad\text{and}\quad
\norm{CT(t)}_{\dot X_1 \to Y} \lesssim 1.
\]
The estimate $Z\to Y$ then follows by interpolation. For
the converse, notice that
\begin{align*}
\norm{ C x }_Y & = c \lrnorm{ \int_0^\infty C AT(2t)x \,\ud t }_Y 
                 \le c  \int_0^\infty\lrnorm{C T(t) AT(t)x }_Y\,\ud t\\
               & \le c' \int_0^\infty t^{-\einhalb} \lrnorm{AT(t)x }_Y\,\ud t
                 \sim c' \norm{ x }_Z.
\end{align*}
Putting both results together we obtain that the Weiss condition {\rm
  (\ref{eq:weiss-condition})} is equivalent to {\rm
  (\ref{eq:norm-estimate})}.  It is clear that \ref{item:decay-rate}
implies \ref{item:weak-admiss}. The remaining implication can be shown
in Hilbert spaces by a dyadic decomposition argument and Fourier
transform. The following quicker and more general argument
has been pointed out to us by Peer Kunstmann: let $\Re(\la)>0$ and
$f(t):=e^{-Re(\la) t}$ on $[0, \infty)$. Then $f^*(t)= f(t)$ since $f$
is decreasing and continuous. One has therefore
\[
\norm{ f }_{L^{2,1}(\RR_+)} = \int_0^\infty t^{\einhalb} f^*(t) \tfrac{\ud t}t 
= \Re(\la)^{-\einhalb} \int_0^\infty s^{-\einhalb} e^{-s} \,\ud s 
= \Re(\la)^{-\einhalb} \Gamma(\einhalb).
\]
Now use the duality $(L^{2,1}(\RR_+))' = L^{2,\infty}(\RR_+)$ (see
\cite[Theorem 1.4.17]{Grafakos})  and the Laplace transform 
 {\rm (\ref{eq:laplace-trafo})}  to conclude that 
\begin{align*}
\norm{ C (\la+A)^{-1}x }_Y 
& \le  \int_0^\infty \bignorm{ C T(t)x }_Y f(t) \,\ud t\\
& \le  \norm{ f }_{L^{2,1}(0, \infty)} \bignorm{  CT(\cdot)x}_{L^{2, \infty}(\RR_+; Y)} \\
& \le C   \Re(\la)^{-\einhalb} \norm{x}. \qedhere
\end{align*}
\end{proof}

\begin{proof}[Proof of Theorem~\ref{thm:optimality-counterexample}]
  Recall that Lorentz spaces satisfy $L^{2,p} \subset L^{2,q}$ for
  $p\le q$. It suffices therefore to show that $C$ is not
  $L^{2,q}$--admissible for any finite $q>2$.  To this end we fix
  $2<q<\infty$. The idea is to extend the counterexample of Jacob and
  Zwart \cite{JacobZwart:counterexable-Weiss-conj} in the following
  sense: not only do we pass from $q=2$ to $q \ge 2$ (recall
  $L^{2,2}=L^2$), but, in contrast with their abstract argument
  referencing to interpolation sequences, we simply provide an
  explicit element $x\in H$ for which the observed orbit $CT(t)x$ does
  not lie in the Lorentz space $L^{2,q}(0,\tau; Y)$. The main idea
  however is identical to the construction of Jacob and Zwart: on
  Hilbert spaces there exist conditional bases $(e_n)$ that satisfy
   \begin{enumerate}
   \item\label{item:nicht-bessel} $(e_n)$ is not Besselian, i.e. there
     is no constant $c_B>0$ such that
\[  
   \sum |\alpha_k|^2 \le c_B \lrnorm{ \sum \alpha_k e_k }. 
\]
\item\label{item:aber-hilbertian} $(e_n)$ is Hilbertian, i.e. there is
  a constant $c_H>0$ such that
\[ 
    \lrnorm{ \sum \alpha_k e_k }  \le c_H \sum |\alpha_k|^2. 
\]
    \item\label{item:inf-positif} $\inf \norm{ e_n} >0$
   \end{enumerate}
   A concrete example is given in \cite[Example
   II.11.2]{Singer:bases1}: let $\beta = \tfrac1{2q'} \in (\tfrac14,
   \tfrac12)$ and consider
\[
e_{2n}(s) = |s|^\beta e^{ins} \quad \text{and}\quad e_{2n+1}(s) =  |s|^\beta e^{-ins} 
\]
on $H=L^2(-\upi, \upi)$. Let $T(t)$ be the exponentially stable and
analytic semigroup given by $T(t) e_n = e^{-4^n t} e_n$ and consider
an observation operator $C: H \to \CC$ given by $C e_n = 2^n$ for $n
\in \NN$. Thus, if $x = \sum \xi_k e_k$,
\[
   | CT(t) x| = \left| \sum_k 2^k e^{-4^k t} \xi_k \right|.
\]

\noindent Let $x(s) = |s|^{-\beta} \eins_{[-\upi,
  \upi]}$. Since $x$ is square integrable, there exist $(\xi_n)$ such
that
\[
    x(s) = \sum_{n\ge 0} \xi_n e_n(s).
\]
However, necessarily, 
\[
\xi_{2n} = \tfrac1{2\upi}\int_{-\upi}^\upi x(s) |s|^{-\beta} e^{-ins}\,\ud s 
   \quad \text{and}\quad 
\xi_{2n+1} = \tfrac1{2\upi} \int_{-\upi}^\upi x(s) |s|^{-\beta} e^{ins}\,\ud s.
\]
Let us determine the growth order of the coefficients: by symmetry,
$\xi_{2n} = \xi_{2n+1}$. Let $\ga>0$. Then 
\[
   \int_0^\upi s^\ga e^{ins}\tfrac{\ud s}{s}  
= n^{-\ga} \int_0^{\upi n} s^{\ga}   e^{is}\tfrac{\ud s}{s} .
\]
The function $\varphi(z) = z^{\ga-1} e^{iz}$ is holomorphic on the
open right half plane. By deplacing the integral from $[0, n\upi]$ to
the positive imaginary axis,
\begin{align*} 
 \int_0^\upi s^\ga e^{ins}\tfrac{\ud s}{s}
 &=n^{-\ga}  \left[ e^{i\upi\ga/2} \int_0^{\upi  n} s^\ga
   e^{-s}\tfrac{\ud s}{s} - \int_0^\pihalbe (n\upi e^{i\theta})^\ga
   e^{in\upi e^{i\theta}}\,d\theta  \right].
\end{align*}
The first integral behaves as $n^{-\ga} e^{i\upi\ga/2} \Gamma(\ga) +
O(\tfrac1n)$;
the second is $O(\tfrac1n)$ as can be seen by estimating $e^{-n\upi \sin(t)}$ using
$\tfrac2\upi t \le \sin(t) \le t$ on $[0,\pihalbe]$. By complex conjugation,
\[
    \int_{-\upi}^0 |s|^\ga e^{ins}\tfrac{\ud s}{s}
 = \int_0^\upi s^\ga e^{-ins}\tfrac{\ud s}{s}
 = \overline{ \int_0^\upi s^\ga e^{ins}\tfrac{\ud s}{s}}.
\]
so that 
\[
   \xi_n = 2 n^{-\ga} \cos(\ga\upi/2) \Gamma(\ga)+O(\tfrac1n).
\]
In our case, $\ga=1-2\beta$, and so $\xi_n \sim n^{2\beta-1} =
n^{-\sfrac1q}$ when $n\to +\infty$.  Finally notice that $\xi_n\ge 0$ since 
\[
\xi_n = 2\sum_{l=0}^{n-1} \int_0^{2\upi} (2\upi l+x)^{\ga-1} \cos(x)\,dx
\]
and each term in the sum is positive : let $f(x) = (2\upi
l+x)^{\ga-1}$. Then $f \in C^\infty(\RR_+)$ is positive, decreasing
and convex. Therefore,  $g(x) = f(x)-f(x{+}\upi)$ satisfies 
$g'(x) = f'(x)-f'(x{+}\upi) = \upi f''(\eta) \le 0$. So $g$ is
decreasing and positive. Thus,
\[
 \int_0^{2\upi} f(x) \cos(x)\,dx = \int_0^\upi \bigl(
 f(x)-f(x{+}\upi)\bigr)  \cos(x)\,dx \ge 0.
\]
Now let us come back to the observed semigroup
orbits: recall that for exponentially stable semigroups admissibility
in arbitrary small finite time or in infinite time are equivalent.
The lack of $L^{2,q}$--integrability must therefore happen near the
origin. For our choice of $x\in H$ and $t \in [4^{-n-1}, 4^{-n})$, one
has
\[
  | CT(t) x|  =  \sum_{k=0}^\infty \xi_k  2^k e^{-4^k t}
     \ge \xi_n  2^n e^{-1} \sim  n^{-\sfrac1q} 2^n \sim
     (1{+}|\log(t)|)^{-\sfrac1q} t^{-\einhalb}  =: f(t).
\]
But $f \not\in L^{2, q}(0, \tau)$ howsoever we choose $\tau\in (0,1)$.
\end{proof}

\medskip

\noindent {\bf Open problem:} 
An inspection of the proof of Theorem
\ref{thm:schwache-L2-zulaessigkeit} shows that the implication
\ref{item:decay-rate} $\Rightarrow$ \ref{item:weak-admiss}
$\Rightarrow$ \ref{item:weiss-cond} always holds. It is only for
\ref{item:weiss-cond} $\Rightarrow$ \ref{item:decay-rate} that
analyticity is used. This hypothesis is not optimal since the
implication is trivially true in all cases when $C$ is even
$L^2$--admissible (i.e. when the Weiss conjecture holds). It follows
that on Hilbert spaces the cases of (a) norm-continuous semigroups,
(b) normal semigroups, (c) exponentially stable right invertible
semigroups and (d) in case $Y=\CC$ contraction semigroups and (e)
diagonal semigroups on a Riesz basis are covered (see
\cite{JacobPartington:survey} for references).  Our analysis of an
non-Riesz basis example underlines further the idea that the Weiss
condition {\rm (\ref{eq:weiss-condition})} and $L^{2,
  \infty}$--admissibility should be equivalent on Hilbert spaces.

\medskip

\noindent We thank Peer Kunstmann and Hans Zwart for pointing out a
calculation error in a preliminary version of this article.


\def\cprime{$'$}
\providecommand{\bysame}{\leavevmode\hbox to3em{\hrulefill}\thinspace}

\end{document}